\documentclass[12pt]{article}
\usepackage{amssymb,amsthm,amsfonts,amsmath,latexsym}
\usepackage[english]{babel}
\usepackage{hyperref}
 \usepackage{color}
\newtheorem{theorem}{Theorem}
\newtheorem{lemma}{Lemma}
\newtheorem{definition}{Definition}
\newtheorem{proposition}{Proposition}
\newtheorem{remark}{Remark}

\newcommand\N{{\mathbb N}}
\newcommand\R{{\mathbb R}}
\renewcommand\P{{\mathbb P}}
\newcommand\ind{{\mathbf 1}}
\newcommand\E{{\mathbb E}}

\newcommand{\Var}{{\rm Var}}
\newcommand{\Const}{{\rm Const}}

\newcommand{\X}{\mathbb X}
\newcommand{\rX}{\boldsymbol X}
\renewcommand{\k}{\boldsymbol k}
\newcommand{\B}{\boldsymbol B}

\providecommand{\keywords}[1]{\textbf{{Keywords}} #1}
\providecommand{\AMS}[2]{\textbf{{Mathematics Subject Classification}} #1}
\begin{document}
\begin{center}
{\LARGE Studying the winding number for  
stationary Gaussian processes using real variables
\footnote{\today}}\\
\vspace{1cm}

\begin{tabular}{cc}
 {\bf J.-M. Aza\"{i}s} & {\bf F. Dalmao} \\
 IMT & DMEL\\
 Universit\'e de Toulouse, France
 & Universidad de la Rep\'{u}blica, Uruguay\\
  {\small jean-marc.azais@math.univ-toulouse.fr}
 & {\small fdalmao@unorte.edu.uy}
\end{tabular}
\vspace{.75cm}

\begin{tabular}{c}
 {\bf J.R. Le\'on} \\
 IMERL\\
 Universidad de la Rep\'{u}blica, Uruguay\\
 Universidad Central de Venezuela, Venezuela\\
 {\small rlramos@fing.edu.uy}
\end{tabular}
\end{center}

\begin{abstract}
We consider the winding number of planar stationary Gaussian processes defined on the line. 
Under mild conditions, we obtain the asymptotic variance and the Central Limit Theorem for the winding number as the time horizon tends to infinity.
In the asymptotic regime, our discrete approach is equivalent to the continuous one 
studied previously in the literature and our main result extends the existing ones.
Our model allows for a general dependence of the coordinates of the process and non-differentiability 
of one of them. 
Furthermore, beyond our general framework, 
we consider as examples an approximation to the winding number of a process 
whose coordinates are both non-differentiable 
and the winding number of a process which is not exactly stationary. 
\end{abstract}

\\
\AMS{ 60G15 . 60G10 }

\section{Introduction}
The notion of the number of winding turns for a planar Gaussian process has been studied for a long time. 
Concerning processes with irregular paths, the argument in the complex plane is defined and studied with the help of the stochastic calculus. 
The asymptotic behaviour of the argument as the time horizon tends to infinity was determined in some important cases 
starting with planar Brownian motion. Among a huge number of results, it is worth
mentioning the seminal theorem of Spitzer \cite{Spitzer}, which states that, with a normalization by $\log t$, the argument of the Brownian curve tends in 
law to a Cauchy distribution. 
Later on, a similar result was obtained for the Ornstein--Uhlenbeck process, see \cite{Va:Va}. 
Note that Messulam and Yor \cite{Mes:Yor} proved that the occurrence of the Cauchy distribution as a limit 
(and consequently, the non-existence of the mean of the limit random variable) is due to the fact that the Brownian motion 
makes many turns when it is close to the origin. 
If the process is restricted to remain outside a neighbourhood of the origin, then the limiting distribution has moments of all orders. 

In the 1980s, intense research was carried out on these matters (the reader may consult the recent paper 
\cite{homenajeYor}, written in memory of Yor,  for a modern view of the existing results and their relations with a deep problem from 
mathematical physics. Another key reference is \cite{le doussal}, which provides a very interesting discussion of recent advances 
and also applications to physics. 
Physical applications involve the physics of polymers, flux lines in 
superconductors, and the quantum Hall effect. 
For more information, the reader can consult the references provided in \cite{le doussal}. 

Following these results for non-differentiable random processes, the study of the winding number was extended  to stationary Gaussian processes 
 whose paths are sufficiently smooth. 
This allows using the classical formula of complex analysis for the argument function to define the winding number. 
With this representation at hand and using techniques for complex-valued processes, Buckley and Feldheim 
\cite{BF} studied the asymptotic behaviour of the variation of the argument $\Delta(T)$, see \cite[Eq. (1)]{BF} and Remark 
\ref{r:Delta} below, of a complex random process $ \X(t) $ following the circularly symmetric model 
over the interval $[0,T]$ as $T\to\infty$. For a definition of this and related classes of processes, see below, in subsection \ref{r:sm}. 
They computed the moments and obtained a central limit theorem.

\smallskip

In the present paper  we use a completely different approach to the same problem, 
counting the number of winding turns around the origin: 
they are counted by the up-crossings minus the down-crossings of the half line $\{x_1 >0, x_2 = 0\}$. 
We will make this definition precise and establish the link 
with the definition  using complex variables in subsection \ref{definition}.

Our approach may seem less intuitive and more involved 
since we replace a continuous functional  
by a discrete one, a priori more difficult to handle. 
Nevertheless, by doing so we can profit from the extensive machinery developed for the study of crossings, such
as Kac--Rice type formulas \cite{marioyluigi}, 
chaos (also known as Wiener--It\^o or Hermite) expansions of level functionals of Gaussian processes, diagram formula \cite[Lem. 3.2]{taqqu} 
and the so called Fourth Moment Theorem and its generalizations \cite{Tu:Pe,Nou:Pec}. 
Actually, our work can be considered as a response to the sentence in \cite{BF}:
`It may well be the case that the more sophisticated methods of Wiener--It\^o expansions could be useful'.
As a matter of fact, 
the complex integral representation of the variation of the argument $\Delta(T)$ in \cite[Eq. (1)]{BF} 
does not seem well adapted to obtain the chaos expansion, see Remark \ref{r:Delta} below. 
On the other hand, the crossings of a half-line are well suited for this purpose. 

Our results are true under wider conditions, 
while, in our opinion, the proofs are simpler. 
Theorem \ref{t:gene} below extends and unifies the existing results, 
such as those of \cite{le doussal} and \cite{BF}. 
In particular, we consider less symmetric models, including general dependencies between the coordinates of the process 
and the case where one of them is not differentiable. 
This shows that real analysis is a very good alternative to complex analysis
in this particular case. 
In addition, the real representation allows, 
in our opinion, obtaining a more explicit description of the different sub-models, 
see subsections \ref{s:gen} and \ref{r:sm}. 

Though we restrict ourselves to the stationary case, 
we point out that both the Kac--Rice formulas and chaos expansions can be used 
in the non-stationary case. 
For instance, in \cite{adl,dnnp}, these techiques are used in a non-stationary framework. 
The main advantage of assuming stationarity is that it implies many symmetries and independences, 
thus simplifying enormously the computations.
 
Finally, let us mention that chaos expansions may allow obtaining a quantitative version of the Central Limit Theorem (CLT), e.g. obtaining bounds for the distance, in a suitbable sense, of the law of the normalized winding turns from the standard Gaussian law. 
This happens to be a quite direct by-product of our forthcoming analysis for finite chaos expansions while the tail of the expansion could be adressed as in \cite[Eq. (4.44)]{NPP}.

\bigskip

This paper is organized as follows. 
Section \ref{s:dom} introduces the model as well as some particular cases. 
Our main result, Theorem \ref{primerth},
is presented in Section \ref{s:mr}. 
The proof is presented in Sections \ref{s:moments}, \ref{CE} and \ref{s:clt}.
Section \ref{s:ex} is dedicated to some examples, 
one of which is not exactly stationary. 
Section \ref{s:ac} contains some auxiliary computations 
used in the proof of the main results.

\section{Description of the model} \label{s:dom}
\subsection{Generalities}\label{s:gen}
Consider a stationary mean-zero vectorial Gaussian process defined on $\R$,
\[
\rX(\cdot)=(X_1(\cdot),X_2(\cdot))\in \R^2. 
\]
Let $r_i(\cdot)$ denote the covariance function of $X_i$, $i=1,2$ and 
$r_{12}(\cdot)$ their cross-covariance function:
\begin{equation*}
r_{i}(t)=\E(X_i(t)X_i(0)), i=1,2;\quad
r_{12}(t)=\E(X_1(t)X_2(0));\quad t\in\R.
\end{equation*}
Without loss of generality, after a spatial scaling, 
we can assume that for each $t$ the variance--covariance matrix of $\rX(t)$ is the identity matrix $I_2$. 
That is, $r_i(0)=1,\,i=1,2$ and $r_{12}(0) = 0$. 
Indeed, 
in Remark \ref{r:ss} it will be justified that a spatial scaling 
plays no role in the asymptotic study of the winding number.
In addition, note that
\begin{equation*}
\E(X_1(0)X_2(t)) = r_{12}(-t);\quad
\E(X_1(0)X'_2(t))=-r'_{12}(-t).
\end{equation*}
Clearly, $r_1,r_2$ and $r_{12}$ determine the distribution of $\rX(\cdot)$.

\medskip 

To be able to compute the winding number we assume that one of the coordinates is differentiable in quadratic mean, say $X_2$. 
The fact that $X_2$ has a derivative is equivalent to  $$-r''_2(0)=:\lambda_{2,2}<\infty.$$
After a scaling in time, we can assume w.l.o.g. that $ \lambda_{2,2}=1$.

\subsection {Some particular models}\label{r:sm}
For the sake of ease of comparison with the literature, and to describe the sub-models, consider now the complex counterpart of $\rX(\cdot)$, 
namely, the centred stationary complex Gaussian process $\X(\cdot)$ s.t.
\[
\X(t):=X_1(t) + \imath X_2(t). 
\]
The distribution of $\X(\cdot)$ is determined by 
its covariance and pseudo-covariance functions, given by
\begin{align*}
R(t) &:= \E\big(\X(0) \overline \X(t)\big)
=r_1(t)+r_2(t)-i(r_{12}(t)-r_{12}(-t));\\
C(t) &:= \E\big(\X(0) \X(t)\big)
=r_1(t)-r_2(t)+i(r_{12}(t)+r_{12}(-t)).
\end{align*}

As said in the Introduction, in the literature 
some symmetries are usually imposed (see \cite{BF,le doussal}). 
We present now the most common sub-models.
\begin{enumerate} 
  \item {\bf The circularly symmetric model.} Assume that $C(t)=0$, or, equivalently, that 
$r_1=r_2$ and that $ r_{12}$ is an odd function (see for instance \cite{BF} 
and \cite[Sec. 8.1,pp. 163]{CL}).
  
  \item {\bf The reflexional symmetric model.} Assume that 
\[
(X_1(\cdot),X_2(\cdot))\stackrel{d}{=} (X_2(\cdot),X_1(\cdot)), 
\]
where $\stackrel{d}{=}$ means that both sides have the same distribution. 
A simple computation shows that $R(t)=2r(t)$ is real and the pseudo-covariance  
is purely imaginary, $C(t)=2ir_{12}(t)$. 
Here, $r$ is the common value of $r_1$ and $r_2$.

 \item {\bf The independent model.} This is the case where $X_1(\cdot)$ and $X_2(\cdot)$ are independent, 
 or, equivalently, where $r_{12} \equiv 0$. 
 
 \item {\bf The i.i.d. model.} Here, 
 $X_1(\cdot)$ and $X_2(\cdot)$ are independent and have the same distribution. 
 This model is often considered by physicists \cite{le doussal}.
\end{enumerate}
The intersection of any two of the models (1), (2), 
and (3) yields model (4). 
  
\subsection{Real variable definition of winding}\label{definition}
Define the number of winding turns around the origin by
\begin{multline} \label{e:nw}
N_W([0,T])
 =\#\{t\le T:\, X_1(t)>0,\, X_2(t)=0,\, X'_2(t)>0\} \\
 -\#\{t\le T:\, X_1(t)>0,\, X_2(t)=0,\, X'_2(t)<0\}.
\end{multline}
Thus, $N_W([0,T])$ is just the number of up-crossings minus the number of down-crossings of $X_2$ 
conditioned on the event $X_1(t)>0$. 

\begin{remark} \label{r:ss}
The choice of the semi-axis $\{ X_1>0,X_2 =0 \}$ is arbitrary: 
we can replace it by any other half-line starting from zero. 
This fact explains why, 
without loss of generality, we can perform the spatial scaling of Section \ref{s:gen}. 
\end{remark}

\begin{remark} \label{r:Delta}
In \cite{BF}, for a complex stationary Gaussian process $\X(t)=X_1(t)+\imath X_2(t)$, 
the increment of the argument is defined by
\begin{equation}\label{rational1}
\Delta(T)=\int_0^T\frac{X'_2(t)X_1(t)-X'_1(t)X_2(t)}{X_1^2(t)+X^2_2(t)}dt. 
\end{equation}
Our approach  is linked to that of the  paper above because 
the increment of the argument $\Delta(T)$ and the winding number $N_W([0,T])$ 
are related by
\begin{equation} \label{e:delta}
 \left| \frac{\Delta(T) }{2 \pi} - N_W([0,T])\right| <1.
\end{equation}
Thus, the difference between $\frac{1}{2 \pi}\Delta(T)$ and $N_W([0,T])$ 
plays no role in our asymptotic study.

In Section \ref{CE}, we will obtain a chaos expansion for $N_W([0,T])$ that allows obtaining a CLT for this random variable and 
consequently also for $\Delta(T)$ because of \eqref{e:delta}.  But one should be aware that  we do not provide any chaos expansion for $\Delta(T)$. We 
point 
out that such a representation seems very difficult to obtain due to the lack of integrability, see Remark \ref{r:Delta2} below,
while the Kac-type integral representation of $N_W([0,T])$ is well adapted to obtain the expansion.
\end{remark}

\section{Main results} \label{s:mr}
Consider the following conditions.
\begin{enumerate}
 \item[(G)] Assume that $X_2(\cdot)$ satisfies  
\begin{equation*} 
 \int \frac{\lambda_{2,2}+r_2''(t)}{t}\,dt \mbox{ converges at zero.}
\end{equation*}
\end{enumerate}
It is well known 
that this condition  is necessary and sufficient for having a finite second factorial moment for the number of 
zeros of $X_2$.  
Geman proved this equivalence in \cite{geman}.

\begin{enumerate}
\item[(A)] Set 
\begin{equation*}
 m(t)=\max\big\{|r_2(t)|,|r''_2(t)|,|r_1(t)|, |r_{12}(t)|,|r'_{12}(t)|\big\},
\end{equation*}
and assume that $m\in\mathbb L^2([0,\infty))$ and $m(t)\to0$ as $t\to\infty$. 

\item[(A')] Assume that 
$r_1(t), r_2(t) \to 0$ as $t\to  +\infty$ and
\begin{equation} \label{e:gene}
\int_\R  r_2^2  + (r'_{12})^2 + (r'_{2})^2 + |r_1r''_2|  < +\infty.
\end{equation}
\end{enumerate}
These mixing conditions differ slightly from the one introduced by Arcones \cite[Lem. 1]{arcones}.

\bigskip

Set
\begin{equation} \label{e:nnw}
N^*_W([0,T]):= \frac{N_W([0,T])-\E N_W([0,T])}{\sqrt{T}}.
\end{equation}
\begin{theorem}\label{primerth}
Consider $\rX(\cdot)$ as in Section \ref{s:dom}. 
Let $N_W([0,T])$ and $N^*_W([0,T])$ be defined 
as in \eqref{e:nw} and \eqref{e:nnw} respectively. 
Hence,
\begin{enumerate} \label{t:gene}

\item For each $T>0$, $\E\big( N_W([0,T]) \big) = -\tfrac{T}{2\pi}r'_{12}(0)$. 
 
 \item Assume that conditions (G) and (A') hold.  
Then, there exists $V_\infty <\infty$ s.t.
\begin{equation}\label{finitevar}
 \lim_{T\to\infty}\frac{\Var \big( N_W([0,T]) \big)}{T} = V_\infty.
\end{equation}

 \item Under conditions (G) and (A), as $T\to\infty$, the distribution of $N^*_W([0,T])$ 
converges towards the centred normal distribution with variance $V_\infty$. 
\end{enumerate}
\end{theorem}

Statements 1, 2  and 3 are proven   in Sections  \ref{s:exp},
\ref{s:var} and   \ref{s:clt}, respectively. 

\medskip

Some remarks are in order. 
\begin{remark} \label{r:tp}
\begin{enumerate}
 
  \item {\bf Finiteness of the expectation and of the variance.} 
  For each $T>0$, $\E N_W([0,T])$ is finite 
  and under $(G)$, $\Var\big( N_W([0,T]) \big)$ is finite, see Section \ref{s:moments}. 

  \item Note that Condition (A') is weaker than Condition (A).
  
  \item {\bf Extension of Theorems 1 and 3 of \cite{BF}. }
  If $r'_1$ exists, $\int|r_1r''_2|<\infty$ can be replaced, in \eqref{e:gene}, by $\int|r'_1r'_2|<\infty$. 
  From this fact we can see that this theorem extends Theorem 3 in \cite{BF}. 
  
  Furthermore, in \eqref{ml:rice2} of Section \ref{s:var} we give an explicit integral formula for the variance of $N_W([0,T])$ for each $T>0$.  This 
 formula is analogous to the one exhibited in Theorem 1 of \cite{BF} but recall that the functional $N_W([0,T])$ does not coincide with $\Delta(T)$ in \cite{BF}.
  
\end{enumerate}
\end{remark}

Under more restrictive hypotheses we can give more precise results. The next proposition, 
whose proof is deferred to Section \ref{s:ac}, 
concerns the positivity of the asymptotic variance.  
It is convenient to consider the condition
\begin{enumerate}
 \item[(S)] Assume that $X_1$ and $X_2$ have spectral densities; 
 they will be denoted by $f_1$ and $f_2$.
\end{enumerate} 
 Note that Condition (S) is weaker than (A) 
 but not weaker than (A').

\begin{proposition}\label{positive}
Under condition (S) and assuming that $r'_{12}(0)=0$, 
we have that $V_\infty>0$. 
\end{proposition}
 
\begin{remark} 
The general case $r'_{12}(0)\neq0$ could be dealt with by the same techniques, but the 
computations involved become burdensome. 
\end{remark}

The next theorem, 
whose proof is deferred to Section \ref{ss:asv},
deals with the simpler case of independent coordinates.
\begin{theorem}\label{t:ind} 
Assume that $X_1$ and $X_2$ are independent 
and that the covariance $r_1(\cdot) $ is differentiable except at the origin. 
Then, a sufficient condition to have a finite asymptotic variance 
(for the r.v. defined in \eqref{e:ind10}) is
\begin{equation*} 
I:=  \int_0^{\infty} \frac{ r'_1(t)}{ \sqrt{1-r_1^2(t)}}\frac{ r'_2(t)}{ \sqrt{1-r_2^2(t)}} dt  
\mbox{ is convergent in the sense of Riemann}.
\end{equation*}
The asymptotic variance takes the value
\begin{equation*} \label{e:ind10}
\lim_{T\to\infty}\frac{\Var(N_W([0,T]))}{T} = \frac 1\pi \Big( \frac{\pi}{2} +  I  \Big).
\end{equation*}
\end{theorem}
In the particular case where the distributions of $X_1$ and $X_2$ are equal 
(the i.i.d. model), the asymptotic variance is equal to
\[
\frac 1\pi \Big( \frac{\pi}{2} +    
\int_0^{\infty}
 \frac{ r'^2(t)}   { 1-r^2(t)} dt 
  \Big),
\]
where $r(\cdot)$ denotes the common value of $r_i(\cdot), i=1,2$.
This expression coincides with the results of \cite{BF} and \cite{le doussal}.

\begin{remark}
Note that the integrability condition in Theorem \ref{t:ind} 
is weaker than condition (G).
\end{remark}

\begin{remark} \label{r:Delta2} 
By defining 
\[
H(x_1,x_2,x_3,x_4)=\frac{(x_4x_1-x_3x_2)}{x^2_1+x^2_2}, 
\]
we can write
\begin{equation*}\label{rational2}
\Delta(T)=\int_0^TH(X_1(t),X_2(t),X'_1(t),X'_2(t))dt.
\end{equation*}
Nevertheless, the CLT of Theorem \ref{e:nnw} 
is not a direct application of 
a continuous time vectorial Breuer--Major Theorem and condition (A), 
because in the present case $H$ does not belong to the space of 
square-integrable functions 
with respect to the four-dimensional standard Gaussian measure.
\end{remark}

\section{Moments} \label{s:moments}
In this section we compute the first two moments of $N_W([0,T])$. 

We start by expressing $N_W([0,T])$ by a Kac-type counting formula, as
\begin{equation}\label{e:kac}
N_W([0,T])
=\lim_{\delta\to0}\frac{1}{2\delta}\int^T_0\mathbf1_{[-\delta,\delta]}(X_2(t))X'_2(t)
\mathbf1_{[0,\infty)}(X_1(t))dt,
\end{equation}
where the limit is in the a.s. sense, 
see \cite[Lem. 3.1, pp 70-71]{marioyluigi} where similar level functionals are treated.

\subsection{The expectation of the winding number} \label{s:exp}
The expectation can be computed by  the Kac--Rice  formula, 
as we do below, but it can also be deduced from  the Hermite expansion, something that  will be 
presented in Section \ref{CE}.

From (\ref{e:kac}) and a proof similar to that of the Kac--Rice formula \cite[Rk 8, pp. 85]{marioyluigi}, we have

\begin{align} \label{e:ENW}
\E(N_W([0,T])) &= \frac T{\sqrt{2\pi}} \E \big( \big[(X'_2)^+  -(X'_2)^- \big] \ind_{X_1>0}|X_2 =0\big) \notag\\
&= \frac{T}{\sqrt{2\pi}} \E \big(X'_2\ind_{X_1>0}| X_2 =0\big) 
=  -\frac{Tr'_{12}(0)}{2\pi}.
\end{align}

Here, $X^+(\cdot)$ (resp. $X^-(\cdot)$) denotes the positive part (resp. negative part) of $X(\cdot)$. 
Note that $r_{12}(t)=\E[X_1(t)X_2(0)]=\E[X_1(0)X_2(-t)]$. We need to consider
$\E[X'_2(0)X_1(0)]$, but stationarity implies $r'_{12}(t) =-\E[X_1(0)X'_2(-t)]=-\E[X_1(t)X'_2(0)]$. Thus $-r'_{12}(0)=\E[X_1(0)X'_2(0)]$. Then 
$-r_{12}(0)$ exists and by the Cauchy--Schwarz inequality is finite.\\
Note that under our hypotheses, the expectation is always finite. 
Note also that in submodels $(2)-(4)$ of Section \ref{r:sm}, it 
vanishes. 
In particular, we have obtained the result of \cite{BF}.

\subsection{The variance of $N_W([0,T])$} \label{s:var}
In this section we assume that $X_2(\cdot)$ satisfies condition (G).
As we said before, see \cite{geman}, this is a necessary and sufficient condition to  ensure that the number  of 
zeros  of $X_2(\cdot)$  has a finite second moment.  
This implies  that $N_W([0,T])$  has finite variance. 

To compute the variance of the random variable $N_W([0,T])$, we use the Kac--Rice formula  
\cite[Rk 8, pp. 85]{marioyluigi} 
and the equality $\Var(N)=\E(N(N-1))-\E^2(N)+\E(N)$. 
We use the short hand notation $\E_c(\cdot)=\E(\cdot\mid X_2(0)=X_2(t)=0)$. 
Hence,  
\begin{multline} \label{ml:rice}
\E\Big(N_W([0,T])(N_W([0,T])-1)\Big)\\
=2\int_0^T(T-t)
\E_c\Big[\mathbf1_{[0,\infty)}(X_1(0))\mathbf1_{[0,\infty)}(X_1(t))X'_2(0)X_2'(t)\Big]
\frac{dt}{2\pi\sqrt{(1-r^2_{2}(t))}}.
\end{multline}

The following lemma, whose proof is postponed to Section \ref{s:ac}, helps us to compute the conditional expectation.
It is a particular case of the celebrated Diagram formula, cf. \cite[Lem. 3.2]{taqqu}.
\begin{lemma}\label{lemme de la mort}
Let $(Z_1,Z_2,Z_3,Z_4)$ be a centred Gaussian vector with  variance $1$ 
and covariances $\rho_{ij},\, 1\leq i<j\leq 4$. 
Then,
\begin{equation*}
\E[Z_1Z_2\mathbf 1_{[0,\infty)}(Z_3)\mathbf 1_{[0,\infty)}(Z_4)]
=\frac{\rho_{12}}{4} 
+\frac{\rho_{12}}{2\pi}  \arcsin (\rho_{34})  
+\frac{\rho_{13}\rho_{24}+\rho_{14}\rho_{23}}{2\pi} \frac 1 {\sqrt{1-\rho_{34}^2}}.
\end{equation*}
\end{lemma}

\medskip

As a consequence, when $\rho_{34}\to0$, we get the  expansion 
\begin{equation}\label{e:equiv}
\E[Z_1Z_2\mathbf 1_{[0,\infty)}(Z_3)\mathbf 1_{[0,\infty)}(Z_4)] 
=  \frac1{2\pi}\big(\rho_{12}\rho_{34}+\rho_{13}\rho_{24} +\rho_{14}\rho_{23}\big)+\frac14\rho_{12} + O (\rho_{34}^2). 
\end{equation}

The next step is to  compute the (conditional) covariances involved in 
the factorial moment $E(N_W([0,T])(N_W([0,T])-1))$. 
This is done in the following lemma, which is a direct consequence of Gaussian Regression, see \cite[Proof of Prop. 4.1, p. 96]{marioyluigi} for a similar computation. 

\begin{lemma}[{\bf The variance--covariance matrix}]\quad \label{lavariance}\\
Set  $r_1, r_2,  r'_2,r''_2$  for $r_1(t), r_2(t),r'_2(t),r''_2(t)$ for short and because of the asymmetry, 
we keep the notation $r_{12}(t)$ and $r_{12}(-t)$. 
The conditional variance--covariance matrix of 
$(X'_2(0), X'_2(t),X_1(0),X_1(t))$ given $X_2(0)=X_2(t)=0$ has the expression
\begin{equation} \label{e:matrice}
\left(\begin{array}{cccc}
1 - \frac{(r'_2)^2}{1-r^2_2} & -r''_2 -\frac{r_2 (r'_2)^2}{1-r^2_2} & -r'_{12}(0) + \frac{ r'_2 r_{12}(-t)}{1-r^2_2}& -r'_{12}(t)  - \frac{ r_2 
r'_2 r_{12}(t)}{1-r^2_2}
 \\
 & 1 - \frac{(r'_2)^2}{1-r^2_2} & -r'_{12}(t)  + \frac{ r_2 r'_2 r_{12}(-t)}{1-r^2_2}     & -r'_{12}(0)+\frac{r'_2 r_{12}(-t)}{1-r^2_2}
 \\
& &  1- \frac{r^2_{12}(-t)}{1-r^2_2}  & r_1 + \frac{r_2 r_{12}(t)r_{12}(-t)}{1-r^2_2} 
\\
& & & 1- \frac{r^2_{12}(-t)}{1-r^2_2} 
 \end{array}\right).
\end{equation}
\end{lemma}

In conclusion, we get 
 \begin{multline} \label{ml:rice2}
 \Var (N_W([0,T]) = 
\E\big(N_W([0,T])(N_W([0,T])-1)\big) +\E(N_W([0,T]) - \E^2(N_W([0,T])  \\
=2\int_0^T(T-t)
 \Big(\frac{\rho_{12}}{4} 
+\frac{\rho_{12}}{2\pi}  \arcsin (\rho_{34})  
+\frac{\rho_{13}\rho_{24}+\rho_{14}\rho_{23}}{2\pi} \frac 1 {\sqrt{1-\rho_{34}^2}} \Big)
\frac{dt}{2\pi\sqrt{(1-r^2_{2}(t))}}
\\
   -\frac{Tr'_{12}(0)}{2\pi}  - \Big(\frac{Tr'_{12}(0)}{2\pi}\Big)^2,
\end{multline}
where the $\rho_{ij}$ are given by the entries of \eqref{e:matrice}.
 
This gives an expression which is analogous to that given in \cite[Th. 1]{BF}.  
We recall that since the studied quantities are not exactly equal, their variances also differ.

\subsection{Asymptotic study of the variance} \label{ss:asv}
In this subsection we study the asymptotic behaviour 
of the variance of the random sequence $N_W([0,T])$, obtaining  point 2 of 
Theorem \ref{primerth} and the result of Theorem \ref{t:ind}.\\
For the sake of readability we study first  the independent model.
 
\subsubsection{Independent case}   \label{independent}
In this case, the complexity of the computations is drastically  simplified. 
Note that  this condition is assumed in \cite{le doussal}. 

\begin{proof}[Proof of Theorem \ref{t:ind}]
Since \eqref{e:ENW} implies that the expectation of $N_W$ vanishes, the variance equals  the second factorial moment 
and is given by the Kac--Rice formula \eqref{ml:rice}.
\begin{align*}
 V_T &:= \frac 1 T \Var \big( N_W([0,T] \big)\\ 
 &= 
 \frac {1}{\pi} \int_0^T \frac{(T-t)}{T\sqrt{1-r_2^2(t)}} \E_c \big(X'_2(0)X'_2(t)\big) \P\{X_1(0) >0; X_1(t) >0\}  dt.
\end{align*}
By Lemma 4.3 of \cite{marioyluigi},
$$ 
\P\{X_1(0) >0; X_1(t) >0\} = \arctan \sqrt{\frac {1+r_1(t)}{1-r_1(t)}}   = \arccos \sqrt{\frac {1-r_1(t)}{2}}.
$$
Using the covariances given in Lemma \ref{lavariance}, we get 
\begin{equation*}
 V_T = -  \frac {1}{\pi}\int_0^T \frac{(T-t)}{T}   \frac{   \Big( r''_2(t)   \big(1-r^2_2(t)\big) +  r_2(t) (r'_2(t))^2\Big)}
 {(1-r_2^2(t))^{3/2}}  \cdot\arccos \sqrt{\frac {1-r_1(t)}{2}} dt.
\end{equation*}
We have the following identities  for $t>0$:
\begin{align*}
 \Big(\arccos \sqrt{\frac {1-r_1(t)}{2}}\Big)' &= \frac{r'_1(t)}{\sqrt{1-r^2_1(t)}}; \\
 \frac{ r''_2(t) (1-r^2_2(t))+ r_2(t)(r'_2(t))^2}  { (1-r_2^2(t))^{3/2}} &=  \bigg(\frac{r'_2(t)}{\sqrt{1-r^2_2(t)}} \bigg)' .
\end{align*}
Now, we set
\begin{align*}
 W_T &:= - \int_0^T   \frac{ r''_2(t)   \big(1-r^2_2(t)\big) +  r_2(t) (r'_2(t))^2}
 {(1-r_2^2(t))^{3/2}}  \arccos \sqrt{\frac {1-r_1(t)}{2}} dt;\\
 w_T &:=\int_0^T  t \frac{ r''_2(t)   \big(1-r^2_2(t)\big) +  r_2(t) (r'_2(t))^2}
 {(1-r_2^2(t))^{3/2}}  \arccos \sqrt{\frac {1-r_1(t)}{2}} dt.
\end{align*}
Thus, 
\[
V_T =  \frac 1 \pi  W_T +  \frac{1}{2 \pi T} w_T. 
\]
By integration by parts, we get that 
\[
W_T= \frac{\pi}{2}  - \frac{ r'_2(T)} {\sqrt{1-r^2_2(T)}}  \arccos \sqrt{\frac {1-r_1(T)}{2}}  
+ \int_0^T \frac{r'_2(t)}{\sqrt{1-r^2_2(t)}}  
\frac{r'_1(t)}{\sqrt{1-r^2_1(t)}}dt,
\]
where we have used that, since the second spectral moment of $X_2$ is finite, 
$r'_2$ exists and the Taylor expansion for $r_2$ 
implies that 
$\frac{r'_2(t)}  {  \sqrt{1-r_2^2(t)}} \to -1 \ \mbox{ as  } t \to 0^+$.

Let us look at the second term. 
Again by integration by parts, 
we get that 
\[
  w_T = \big[- tW_t \big]_0^T + \int_0^T W_t dt .
\]
We are now in a position to prove that for all $T$, the variance is finite. 
Consider first $T$ sufficiently small that
$ 1-r_2^2(t)$ and $ 1-r_2^2(T)$ are  bounded away from zero. 
The calculation above proves that the variance $ \Var \big( N_W([0,T] \big) = T\cdot  V_T$ is finite.  
For the second step, we apply 
the  Minkowsky inequality  to get that the variance $ \Var \big( N_W([0,T] \big)$  is  finite for all $T$.  

We now study the asymptotic behaviour  of $V_T$ as $T \to \infty$. 
Under (S) and applying the Riemann--Lebesgue lemma we get that $W_T - I$ converges 
to $\pi/2$.
Hence,
\[
 \lim_{T \to \infty} \frac{w_T}{T} = -W_\infty  + W_\infty   =0,
\]
proving that
\[ 
\lim_{T \to \infty} V_T=\int_0^\infty \frac{r'_2(t)}{\sqrt{1-r^2_2(t)}}  
\frac{r'_1(t)}{\sqrt{1-r^2_1(t)}}dt. 
\]
The result follows.
\end{proof}

\subsubsection{General case}
This section deals mainly with the proof of point $2$ of Theorem \ref{primerth}. 
In the proof we will use systematically the results 
of Lemma \ref{lavariance}.\\
Writing $N_W$ for $N_W([0,T])$, we have 
\[
 \frac{\Var(N_W) }{T} = \frac{
 \E(N_W (N_W-1)) +  \E(N_W) - (\E(N_W))^2}{T}.
\]
Since $\E(N_W) = -\frac{r'_{12}(0)}{2\pi} T$,
\begin{equation} \label{e: jma:var}
  \frac{\Var(N_W) }{T} =-\frac{r' _{12}(0)}{2\pi}+\frac 2 {(2\pi)^2}
 \int_0^T  \frac{T-t}{T} \Big(
 \frac{2\pi E_c}{\sqrt{1-r_2^2}}  - ( r'_{12}(0))^2\Big)dt,
\end{equation}
with $E_c :=
\E_c\Big[\mathbf1_{[0,\infty)}(X_1(0))\mathbf1_{[0,\infty)}(X_1(t))X'_2(0)X_2'(t)\Big]$.

Our next goal is to apply Lemma \ref{lemme de la mort}, 
with the law of the Gaussian vector $(Z_1,Z_2,Z_3,Z_4)$ being equal to the 
conditional law of $(X'_2(0), X'_2(t),X_1(0),X_1(t)))$ given $X_2(0)=X_2(t)=0.$\\
As we have assumed that  $r_1(t), r_2(t) \to 0$ as $t\to  +\infty$ and  Condition (A'), we get
\[
\rho_{34}  =  r_1 + \frac{r_2 r_{12}(t)r_{12}(-t)}{1-r^2_2}  \to 0.
\]
Hence, we can apply \eqref{e:equiv} to get the following terms:

\begin{itemize}
\item $\rho_{12}$. We have
\[
\rho_{12} = -r_2'' - \frac{r_2 r_2'^2}{1-r_2^2}.
\]
A sufficient condition to ensure the finiteness of the contribution of the second term is 
$\int |r_2| (r_2')^2 < +\infty$.

As for the  first one, its contribution is (up to a multiplicative constant that plays no role) 
\[
\frac 1 T \int _0^T ds  \int _0^T - \frac{ r_2''(t-s)}{ \sqrt{1 - r_2^2(t-s)}} dt.
\]
By a first integration by parts  we get that this  quantity is equal  to
\[
  I_1:=\frac 1 T \int _0^T
 \frac{ -r_2'(-s)}{ \sqrt{1 - r_2^2(-s)}}  + \frac{ r_2''(T-s)}{ \sqrt{1 - r_2^2(T-s)}} 
 ds
\]
plus a term which is convergent as long as $ \int |r_2|(r_2')^2 < +\infty$. 
We perform now a second integration by parts to get that 
\[
I_1 = \frac 1 T \bigg( \Big[  \frac{ -r_2'(T-s)}{ \sqrt{1 - r_2^2(T-s)}}  \Big]_0^T   +
\Big[  \frac{ -r_2'(-s)}{ \sqrt{1 - r_2^2(-s)}}  \Big]_0^T \bigg) +I_2.
\]
The first term clearly tends to zero. 
The integral $I_2$ is convergent as long as $ \int r_2^2|r_2'| < +\infty$.

A sufficient (simpler) condition for the convergence of the integral of $\rho_{12}$ is
\[
 \int r_2^2+(r_2')^2<\infty.
\]

\item $\rho_{12} \rho_{34}$. We have
\[ 
 \rho_{12}\rho_{34}  =  \Big(-r_2'' - \frac{r_2 r_2'^2}{1-r_2^2}\Big) \Big( r_1 + \frac{r_2 r_{12}(t)r_{12}(-t)}{1-r^2_2} \Big).
\]
Integrating by parts the term involving $r_2 r_2''$ we get that it is integrable as long as
\[
 \int (r_2')^2<\infty, \quad   \int |r_2''| |r_1|<\infty.
\]

\item $\rho_{13} \rho_{24}$.  We have 
\[
 \rho_{13} \rho_{24}=( r'_{12}(0))^2 -2 \frac{ r'_{12}(0) r'_2 r_{12}(-t)}{1-r^2_2}   +
\frac{ ( r'_2)^2 (r_{12}(-t))^2}
{(1-r^2_2)^2} .
\]
The first term is compensated for by the term $-( r'_{12}(0))^2$ appearing in \eqref{e: jma:var}. Since
\[
\frac 1 {\sqrt{1-r_2^2}} \simeq 1 + \frac{r^2_2}{2},
\] 
a term appears that is equivalent to 
\[
\Const\ r^2_2.
\]
A sufficient condition for the convergence of the remaining terms  is
\begin{equation*} 
\int (r_2')^2 < +\infty, \quad  \int |r_2' r_{12}(-t)| < +\infty
\end{equation*}

\item $\rho_{14} \rho_{23}$. We have 
\[
\rho_{14} \rho_{23} = (r'_{12}(t))^2 - \frac{r_2^2 r_2'^2 (r_{12}(t))^2}{(1-r_2^2)^2}
\leq  (r'_{12}(t))^2;
\]
a sufficient condition for this is 
\begin{equation*} 
\int   (r'_{12}(t))^2 dt< +\infty.
\end{equation*}
\end{itemize}
Gathering all these conditions together, the result \eqref{finitevar} follows.

\section{Chaos expansion of the number of winding turns} \label{CE}
In order to prove the asymptotic normality of the standardized winding number, 
we need to work with Hermite polynomials:  
they are defined by
\[
H_n(x)=(-1)^n\frac{d^n}{dx^n}(e^{-\frac12x^2})e^{\frac12 x^2},\quad x\in\R, n\geq 0.
\]
They form a complete orthogonal system in the space $\mathbb L^2(\R,\phi(x)dx)$ of square integrable 
functions with respect to the standard Gaussian measure $\phi(x) dx$. 
One of the key properties of Hermite polynomials, 
known as Mehler's formula, establishes that  
for a vector $(X,Y)$ of standard Gaussians with correlation $\rho$,
\[
\E[H_n(X)H_m(Y)]=\delta_{n,m}n!\rho^n.
\]

\medskip

We now give the Hermite expansion for $N_W([0,T])$. 
The proof is similar to the analogous expansions in \cite{K:L} and \cite{Slud}.
Recall that in \eqref{e:kac}, $N_W([0,T])$ is written a.s. as
\begin{equation*}
N_W([0,T])=\lim_{\delta\to0}\frac{1}{2\delta}\int^T_0\mathbf1_{[-\delta,\delta]}(X_2(t))X'_2(t)
\mathbf1_{[0,\infty)}(X_1(t))dt.
\end{equation*}

\medskip

In order to take advantage of the independence, 
we perform a regression of $X_1(t)$ on $X'_2(t)$ 
(note that $X_1(t)$ and $X_2(t)$ are independent since $r_{12}(0)=0$).
Thus, we write for each $t\in[0,\infty)$
\begin{equation}\label{mgc}
 X_1(t)=\rho_1X'_2(t)+\rho_2Z(t),
\end{equation}
with $\rho_1=r'_{12}(0)$, $\rho_2=\sqrt{1-\rho^2_1}$ ($\rho_2\neq0$ to avoid trivialities) 
and $Z(t)$ a standard Gaussian r.v. independent from 
$X_2(t),X'_2(t)$. Note that  if $\rho_2=0$, the number of winding turns  is simply equal to  the number of up-crossings  of $X_2(\cdot)$.
Note  also that 
\begin{equation*}
 r_Z(t)=\frac{1}{\rho^2_2}r_1(t)-\frac{\rho_1}{\rho^2_2}(r'_{12}(t)-r'_{12}(-t))-\frac{\rho^2_1}{\rho^2_2}r''_2(t). 
\end{equation*}
Set 
\begin{equation}\label{G}
g(x',z)=x'\ind_{[0,\infty)}(\rho_1x'+\rho_2z)\in L^2(\R^2,\phi_2(dx)),
\end{equation} 
with $\phi_2$ the standard Gaussian density in $\R^2$. 

Put $\k:=(k_1,k_2,k_3)$ and $|\k|=\sum_ik_i$. 
We have
\begin{align}\label{e:he}
 N^*_W([0,T])&=
 \sum^\infty_{q=1} I_{q}(T);\notag\\
 I_q(T)
 &=\sum_{|\k|=q}\frac{a_{k_1}d_{k_2,k_3}}{\sqrt{T}}
\int^{T}_{0} H_{k_1}(X_2(t)) H_{k_2}(X'_2(t))H_{k_3}(Z(t))dt,
\end{align}
where the coefficients $a_{k_1}$ are the Hermite coefficients of the \emph{Dirac delta distribution}, \cite{Slud}, 
\begin{equation} \label{e:dirac}
 a_{2k_1}=\frac{H_{2k_1}(0)}{(2k_1)!\sqrt{2\pi}}=\frac{(-1)^{k_1}(2k_1-1)!!}{\sqrt{2\pi}(2k_1)!} 
= \frac{(-1)^{k_1}}{\sqrt{2\pi}2^{k_1}k_1!}
,\,k_1\ge0;
\end{equation}
and $a_{k_1}=0$ if $k_1$ is odd. We remark that 
\begin{eqnarray}\label{delta}
a^2_{2k_1}(2k_1)!\le \Const.
\end{eqnarray}
In addition, the $d_{k_2,k_3}$ are the Hermite coefficients of $g$, i.e.
\begin{equation} \label{e:d}
 d_{k_2,k_3}=\frac{1}{k_2!k_3!}\int_{\R^2}g(x',z)H_{k_2}(x')H_{k_3}(z)\phi_2(x',z)dx'dz.
\end{equation}
Moreover, since $\{\frac{1}{\sqrt{n!}}H_n(\cdot):n\geq 0\}$ form an orthonormal basis of $L^2(\phi(x)dx)$, 
\begin{eqnarray}\label{norm}
||g(\cdot,\cdot)||^2=\sum_{q=0}^\infty\sum_{k_2+k_3=q}d^2_{k_2,k_3}k_2!k_3!<\infty.
\end{eqnarray}

\medskip

Now, we write this Hermite expansion as a Wiener Chaos expansion, that is, 
we write $I_q(T)$ as a multiple stochastic integral w.r.t. a standard Brownian motion 
$\B=\{B(\lambda):\lambda\in[0,\infty)\}$. 

Since condition (S) holds, for $i=1,2$, $X_i$ has spectral density $f_i$. 
Thus, we have the spectral representation
\begin{equation*}
 X_2(t)=\int^\infty_0 \cos(t\lambda)\sqrt{f_2(\lambda)}dB(\lambda);\quad
 X'_2(t)=-\int^\infty_0 \sin(t\lambda)\lambda\sqrt{f_2(\lambda)}dB(\lambda).
\end{equation*}
It is easy to get a similar representation for $Z$ using the same Brownian motion $\B$.

For $t\in[0,\infty)$, let
\begin{equation*}
 Y(t)=(Y_1(t),Y_2(t),Y_3(t)):=(X_2(t),X'_2(t),Z(t)),
\end{equation*}
and, for $i=1,2,3$, let $\varphi_{i,t}\in L^2([0,\infty))$ be such that
\begin{equation*}
 Y_i(t)=\int^{\infty}_0\varphi_{i,t}(\lambda)dB(\lambda)
 =:I^{\B}_1(\varphi_{i,t}),
\end{equation*}
and
\begin{equation*}
 \E(I^{\B}_1(\varphi_{i,t})I^{\B}_1(\varphi_{j,t'}))
 =\left\langle \varphi_{i,t},\varphi_{j,t'}\right\rangle_{L^2([0,\infty))}.
\end{equation*}

By the properties of Hermite polynomials and stochastic integrals,
\begin{equation*}
 H_{k_1}(X_2(t)) H_{k_2}(X'_2(t)) H_{k_3}(Z(t))\\
 =I^{\B}_{q}\left(\varphi_{1,t}^{\otimes k_1}\otimes\varphi_{2,t}^{\otimes k_2}\otimes\varphi_{3,t}^{\otimes k_3}\right).
\end{equation*}
Hence,
\begin{equation}\label{e:ec}
 N^*_W([0,T])=\\
 \sum^\infty_{q=1}I^{\B}_{q}\left(\sum_{|\k|=q}\frac{a_{k_1}d_{k_2,k_3}}{\sqrt{T}}
 \int^{T}_{0} \varphi_{1,t}^{\otimes k_1}\otimes\varphi_{2,t}^{\otimes k_2}\otimes\varphi_{3,t}^{\otimes k_3}dt\right).
\end{equation}

\section{Central Limit Theorem} \label{s:clt}
We now prove part 3 of Theorem \ref{primerth}, 
the CLT for $N_W([0,T])$. 
We use \cite[Th. 6.3.1, pp.~125--126]{Nou:Pec} and \cite[Th.1]{Tu:Pe}. 
These theorems are extensions of the so called Fourth Moment Theorem \cite[Th. 5.2.7, pp.~99--100]{Nou:Pec}. 
They provide a simple and powerful characterization of the CLT based on the chaos decomposition. 
We will proceed in two steps: 
the first one proves that  the variance of 
$\pi^Q(N^*_W([0,T])) := \sum_{q\geq Q}I_{q}(T)$ is negligible for $Q$ large enough; 
the second step establishes the asymptotic normality of 
$\pi_Q(N^*_W([0,T])) := \sum_{1\leq q\leq Q}I_{q}(T)$. \\

One of the main tools  of the proof is Arcones's inequality  (see \cite{arcones}), 
which is used  to show the asymptotic negligibility of the tail of the expansion. 
For completeness, we give here a statement of this inequality  adapted to our framework. 
We restrict 
the inequality to two three-dimensional standard Gaussian random vectors $\mathcal Z=(Z_1,Z_2,Z_3)$ and $\mathcal W=(W_1,W_2,W_3)$. 
For $1\le j,k\le 3$, set $\gamma_{ij}=\E[Z_iW_j]$ and set
\[
\psi:=\sup_{1\le i\le 3}\sum_{j=1}^3|\gamma_{ij}|\vee \sup_{1\le j\le 3}\sum_{i=1}^3|\gamma_{ij}|.
\]

Let $F$ be a function 
s.t. $\E(F(\mathcal Z))=\E(F(\mathcal W))=0$ and $\|F\|^2:=\E(F(\mathcal Z)^2)<\infty$. 
We consider its expansion in the Hermite basis 
\[
F(x_1,x_2,x_3)=\sum_{(\sum_{i=1}^3k_i)\ge \tau}F_{k_1,k_2,k_3} 
H_{k_1}(x_1)H_{k_2}(x_2)H_{k_3}(x_3). 
\]
Hence, if $\psi\leq 1$, Arcones' inequality can be written as 
\begin{equation}\label{arcones}
|\E(F(\mathcal Z)F(\mathcal W))|\le \psi^\tau\E(F^2(\mathcal Z))=\psi^\tau ||F||^2.
\end{equation}

{\bf Step 1:}
 Here, it is convenient to use the expansion \eqref{e:he}. 
Consider first w.l.o.g. $T=n\in \N$. 
In fact, if $T>0$, write 
$N_W([0,T])=N_W([0,\left\lfloor T \right\rfloor])+N_W([T,T-\left\lfloor T \right\rfloor])$. 
Clearly, for $s\in[0,1]$, $\Var(N_W([0,s]))$ is finite and continuos w.r.t $s$, 
thus $\Var(N_W([T,T-\left\lfloor T \right\rfloor]))\leq \Const$. 
Hence, the second term does not contribute.  

Let us denote by 
\[
Y^Q_{i}=\sum_{q=Q+1}^{\infty}\sum_{|\k|=q}
a_{k_1}d_{k_2,k_3}
\int_i^{i+1} 
H_{k_1}(X_2(t)) H_{k_2}(X'_{2}(t))H_{k_3}(Z(t))dt.
\]
Then,
\begin{multline*}
\pi^Q(N^*_W([0,n]))
=\sum_{q=Q+1}^{\infty}\sum_{|\k|=q}\frac{a_{k_1}d_{k_2,k_3}}{\sqrt{n}}\int_0^n
H_{k_1}(X_2(t)) H_{k_2}(X'_{2}(t))H_{k_3}(Z(t))dt\\
\\=\sum_{i=0}^{n-1}\sum_{q=Q+1}^{\infty}\sum_{|\k|=q}
\frac{a_{k_1}d_{k_2,k_3}}{\sqrt{n}}
\int_i^{i+1} 
H_{k_1}(X_2(t)) H_{k_2}(X'_{2}(t))H_{k_3}(Z(t))dt
=\frac{1}{\sqrt{n}}\sum_{i=0}^{n-1}Y^Q_{i}.
\end{multline*} 
The variance of this random variable is equal to
\begin{equation}\label{e:sum}
\frac1n\sum_{i=1}^n\sum_{j=1}^n\E[Y^Q_{i}Y^Q_{j}]
=\frac1n\sum_{|i-j|\le a}\E[Y^Q_{i}Y^Q_{j}]+\frac1n\sum_{|i-j|> a}\E[Y^Q_{i}Y^Q_{j}],
\end{equation}
where 
$a>1$ is a constant that will be chosen later on. Using the Cauchy--Schwarz inequality we easily get 
\begin{align*}
\left|\frac1n\sum_{|i-j|\le a}\E[Y^Q_{i}Y^Q_{j}]\right|
&\le \frac{\#\{(i,j):|i-j|\le a\}}n\E[(Y^Q_1)^2]\\
&\le \frac{\#\{(i,j):|i-j|\le a\}}n\E[(N_W[0,1])^2]\to0.
\end{align*}
The above uses the stationarity of the process and the fact that $\E[(Y^Q_{1})^2]\le \E[(Y^0_{1})^2]\le \E[(N_W[0,1])^2]$. 
Moreover
\[
\mathcal I_{n,Q}:=\frac1n\sum_{|i-j|> a}\E[Y^Q_{i}Y^Q_{j}]=\frac1n\sum_{q=Q+1}^\infty\sum_{|i-j|> a}\int_{i}^{i+1}\int_{j}^{j+1}\E[F_q(t)F_q(s)]dtds. 
\]
Here, 
\[
F_q(t)=\sum_{|\k|=q}a_{k_1}d_{k_2,k_3} H_{k_1}(X_2(t)) H_{k_2}(X'_2(t))H_{k_3}(Z(t)).
\]
Assume now that $j>i$. Then $|s-t|>a-1$. By Arcones' inequality \eqref{arcones},
\[
|\E[F_q(t)F_q(s)]|\le  \psi^q(|t-s|) ||F_q||^2.
\]
Here, 
\begin{equation*}
\psi(t)
=\sup_{1\leq i\leq 3}\left\{\sum^3_{j=1}|\E(\bar{X}_i(t)\bar{X}_j(0))|
\right\}
\leq \Const\ m(t),
\end{equation*}
where we have set $\bar{X}_1(\cdot)=X_2(\cdot)$, $\bar{X}_2(\cdot)=X'_2(\cdot)$ 
and $\bar{X}_3(\cdot)=Z(\cdot)$. 

Finally, by using (\ref{delta}) and (\ref{norm}),
\begin{eqnarray*}
||F_q||^2
\le \Const\ ||g(\cdot,\cdot)||^2 \ ,
\end{eqnarray*}
where $g(\cdot,\cdot)$ was defined in (\ref{G}).
Choose $\rho>1$ and $a$ such that $\psi(a-1)<\rho<1$. Then
\begin{align*}
|\mathcal I_{n,Q}|
&\le \sum_{q=Q+1}^\infty \rho^{q-1}||F_q||^2\frac1n\sum_{|i-j|> a}\int_{i}^{i+1}\int_{j}^{j+1}\psi(t-s)dtds\\
&\le 2\sup_q||F_q||^2\frac{\rho^Q}{1-\rho}\int_0^\infty\psi(s)ds.\to0\mbox{ as }Q\to\infty.
\end{align*} 

This result implies that the weak convergence of $\pi_Q(N^*_{W}([0,T]))$ implies that of the $N^*_{W}([0,T])$.

\medskip

{\bf Step 2:} 
Theorem 1 in \cite{Tu:Pe} says us that it suffices to state the 
convergence towards a Gaussian r.v. of each term $I_q(T)$. 

We consider separately the term $q=1$. 
We know from \eqref{e:dirac} that $a_1=0$.
Now, routine computations show that $d_{0,1}=0$, and so
\begin{equation*}
 I_1(T)=\frac{a_0d_{1,0}}{\sqrt{T}}\int^T_0X'_2(s)ds.
\end{equation*}
Thus, $I_{1}(T)$ is centred Gaussian with
\begin{equation*}
 \Var(I_{1}(T)) = \frac{a^2_0d^2_{1,0}}{T}\int^T_0\int^T_0\E(X'_2(s)X'_2(t))dsdt
 = \frac{a^2_0d^2_{1,0}}{T}\int^T_0\int^T_0 r''_2(t-s)dsdt
 \mathop{\to}\limits_{T\to\infty}0.
\end{equation*}
Hence, the term $q=1$ converges weakly to Dirac's distribution $\delta_0$, which is a Gaussian with variance equal to zero. 
\medskip

Now, fix $q\geq 2$ and $\k\in\N^3$ s.t. $|\k|=q$. 
Here, it is convenient to use  Expansion \ref{e:ec}.  
Now consider
\begin{align*}
 J_{q,\k}(T)&= I^{\B}_{q}(g_{q,\k,T}),\\
 g_{q,\k,T}&=\frac{1}{\sqrt{T}}\int^{T}_{0} \varphi_{1,t}^{\otimes k_1}\otimes\varphi_{2,t}^{\otimes k_2}\otimes\varphi_{3,t}^{\otimes k_3}dt.
\end{align*}
We also define the symmetrized kernels
\begin{equation*}
 \tilde{g}_{q,\k,T}(\lambda_1,\dots,\lambda_q)=
 \frac{1}{q!}\sum_{\sigma\in{\cal S}_q} g_{q,\k,T}(\lambda_{\sigma(1)},\dots,\lambda_{\sigma(q)}),
\end{equation*}
where ${\cal S}_q$ is the set of permutations of $q$ elements.

Using the Fourth Moment Theorem (see \cite{Nou:Pec} and \cite{Tu:Pe} again) 
in order to prove the asymptotic normality of 
$J_{q,\k}(T)$ as $T\to\infty$, 
it suffices to prove that for $n=1,\dots,q-1$, the $L^2$-norm of 
the so-called contractions
\begin{multline*}
 \tilde{g}_{q,\k,T}\otimes_n \tilde{g}_{q,\k,T}(\lambda_1,\dots,\lambda_{2q-2n})\\
 =\int_{[0,\infty)^n}\tilde{g}_{q,\k,T}(\lambda_1,\dots,\lambda_{q-n};z_1,\dots,z_n)
 \tilde{g}_{q,\k,T}(\lambda_{q-n+1},\dots,\lambda_{2q-2n};z_1,\dots,z_n)
 dz_1\dots dz_n
\end{multline*}
tend to $0$ as $T\to\infty$. 
We show this fact in the rest of this step.

\medskip

For ease of notation, we rename the kernels and their arguments in the following way:
\begin{equation*}
 g_{q,\k,T}
 =\frac{1}{\sqrt{T}}\int^{T}_{0} \otimes^{q}_{i=1}\psi_{i,t}dt,
\end{equation*}
where we set $\psi_{i,t}=\varphi_{1,t}$ for $i=1,\dots,k_1$;
$\psi_{i,t}=\varphi_{2,t}$ for $i=k_1+1,\dots,k_1+k_2$ 
and $\psi_{i,t}=\varphi_{3,t}$ for $i=k_1+k_2+1,\dots,q$. 
Write also 
\[
 (x_1,\dots,x_{q})=(\lambda_1,\dots,\lambda_{q-n};z_1,\dots,z_n); 
 \textrm{ and }
 (y_1,\dots,y_{q})=(\lambda_{q-n+1},\lambda_{2q-2n};z_1,\dots,z_n).
\]
Hence,
\begin{multline*}
 \tilde{g}_{q,\k,T}\otimes_n \tilde{g}_{q,\k,T}(x_1,\dots,x_{q-n};y_1,\dots,y_{q-n})\\
 =\frac{1}{T(q!)^2}\sum_{\sigma,\sigma'\in{\cal S}_q}
 \int^T_0\int^T_0\int_{[0,\infty)^n}
 \otimes^{q}_{i=1}\big[\psi_{i,t}(x_{\sigma(i)})
 \otimes\psi_{i,t'}(y_{\sigma'(i)})\big]
 dz_{1}\dots dz_{n}dtdt'\\
 =\frac{1}{T(q!)^2}\sum_{\sigma,\sigma'\in{\cal S}_q}
 \int^T_0\int^T_0 \prod^n_{j=1}\int_{0}^{\infty}
 \psi_{\sigma^{-1}(q-n+j),t}(z_{j})\psi_{\sigma'^{-1}(q-n+j),t'}(z_{j})dz_{j}\\
 \cdot\otimes^{q-n}_{i=1}\big[\psi_{\sigma^{-1}(i),t}(x_{i})
 \otimes\psi_{\sigma'^{-1}(i),t'}(y_{i})\big]
 dtdt'\\
 =\frac{1}{T(q!)^2}\sum_{\sigma,\sigma'\in{\cal S}_q}
 \int^T_0\int^T_0 \prod^n_{j=1}
 \left\langle \psi_{\sigma^{-1}(q-n+j),t}, \psi_{\sigma'^{-1}(q-n+j),t'}
 \right\rangle_{L^2([0,\infty))}\\
 \cdot\otimes^{q-n}_{i=1}\big[\psi_{\sigma^{-1}(i),t}(x_{i})
 \otimes\psi_{\sigma'^{-1}(i),t'}(y_{i})\big]
 dtdt'. 
\end{multline*}
In the second equality, we used the fact that $z_j=x_{q-n+j}=y_{q-n+j}$. 
By the isometric property of the stochastic integrals, each inner product 
in the above integral 
equals the covariance of the r.v.'s associated to the corresponding kernels, 
namely, the covariance between some of $X_2, X'_2, Z$ at $t$ and $t'$. 

Analogously one sees that when taking the $L^2$-norm of 
$\tilde{g}_{q,\k,T}\otimes_n \tilde{g}_{q,\k,T}$ one gets the integral 
of the product of $2q$ covariances of the same r.v.'s. 
Hence, 
\begin{equation*}
 \|\tilde{g}_{1,\k,T}\otimes_n \tilde{g}_{1,\k,T}\|^2_{L^2([0,\infty)^{2})}
 \leq \frac{1}{T^2}\int_{[0,T]^4}m(t-t')^n m(s-s')^n m(t-s)^{q-n} m(t'-s')^{q-n}
 dsds'dtdt'.
\end{equation*}
Here, we bounded the absolute value of each covariance by $m$, the function defined in condition (A).

We consider the most difficult case: $n=1,q-n=1$ ($q=2$), 
which involves the lowest powers of $m$.
The remaining cases  are easier or analogous to this one.  
Hence, 
\begin{equation*}
 \|\tilde{g}_{1,\k,T}\otimes_1 \tilde{g}_{1,\k,T}\|^2_{L^2([0,\infty)^{2})}
 \leq \frac{1}{T^2}\int_{[0,T]^4}m(t-t')m(s-s')m(t-s)m(t'-s')
 dsds'dtdt'.
\end{equation*}

Consider the isometric change of variables 
$(u_1,u_2,u_3,u_4)\mapsto (t-t',s-s',t-s,t')$, 
thus
\begin{multline*}
\|\tilde{g}_{1,\k,T}\otimes_1 \tilde{g}_{1,\k,T}\|^2_{L^2([0,\infty)^{2})}\\
 \leq \frac{\Const}{T^2}\int_{[0,T]^4}m(u_1)m(u_2)m(u_3)m(u_2-u_1-u_3) \ind_{\{u_2-u_1-u_3\geq 0\}}
 du_1du_2du_3du_4,\\
 \leq \frac{\Const}{T}\int_{[0,T]^3}m(u_1)m(u_2)m(u_3)m(u_2-u_1-u_3)\ind_{\{u_2-u_1-u_3\geq 0\}}
 du_1du_2du_3. 
\end{multline*}
Now, since $m\in \mathbb L^2$,
\begin{equation*}
 \int^\infty_0 m(u_3)m(u_2-u_1-u_3) \ind_{\{u_2-u_1-u_3\geq 0\}} du_3 \leq
\|m\|^2_{L^2[0,\infty)}.
\end{equation*}
Besides, we claim
\begin{equation*}
 \frac{1}{\sqrt{T}}\int^T_0 m(u) du\to_{T\to\infty}0.
\end{equation*}

Indeed, consider $\varepsilon>0$ and $a$ such that $\int_a^{\infty}m^2(u)du\le \varepsilon^2$. Then,
\begin{multline*} \frac{1}{\sqrt{T}}\int^T_0 m(u) du\le ||m||_{\infty}\frac a{\sqrt T}+\frac1{\sqrt T}\int_a^Tm(u)du\\ \le ||m||_{\infty}\frac a{\sqrt T}+\frac{(T-a)^{\frac12}}{\sqrt T}(\int_a^{\infty}m^2(u)du)^{\frac12}
\le||m||_{\infty}\frac a{\sqrt T}+\varepsilon.\end{multline*}
 Letting $T\to\infty$ we get
$\displaystyle\limsup_{t\to\infty}\frac{1}{\sqrt{T}}\int^T_0 m(u) du<\varepsilon,$  and the claim follows.

These bounds prove that
\begin{equation*}
 \|\tilde{g}_{1,\k,T}\otimes_1 \tilde{g}_{1,\k,T}\|_{\mathbb L^2([0,\infty)^{2})}\to_{T\to\infty}0,
\end{equation*}
as we claimed. 
This completes the proof of the CLT.

\section{Examples} \label{s:ex}
We present four examples. The first and the third examples can not be obtained by other techniques, because they concern non-differentiable processes. 
In the second example, the conditions for the CLT are very simple. 
In our last example, we consider a process which slightly escapes from stationarity.

\subsection{Bargmann--Fock and irregular processes}
Assume that $X_2$ is a Bargmann--Fock process 
and that $X_1$ is an Ornstein--Uhlenbeck process 
independent from $X_2$, namely, for $t\geq 0$:
\begin{align*}
r_1(t)& = \exp(-t),\quad r_{12}(t)=0, \\
r_2(t) &= \exp(-t^2/2).
\end{align*}
In this case we know from (\ref{e:ENW}) that 
$\E(N_W([0,T])) =0$ and that the asymptotic variance is given by
\begin{equation*} 
\lim_{T\to\infty}\frac{\Var(N_W([0,T]))}{T} = 
\frac 1\pi \Big( \frac{\pi}{2} +   \int_0^{\infty} \frac{ e^{-t}}{ \sqrt{1-e^{-2t}}}
\frac{ t e^{-t^2/2}}{ \sqrt{1-  e^{-t^2}}} dt 
 \Big).
\end{equation*}
The convergence of the integral at $+\infty$ is direct. As for the convegence at zero, the equivalent of the integrand  is $t^{-1/2}$ that ensures 
convergence.
  
As a consequence, the CLT  holds. 
Note that this example is out of reach of other methods.

\medskip
 
To generalize this example we need a definition. 
\begin{definition}  \label{alpha} 
Let $0<\alpha< 2$. 
We define an $\alpha$-process as a stationary Gaussian process 
with a covariance $\rho(t)$ that satisfies 
\begin{itemize}
\item 
$ \rho(t) =  1-C t ^\alpha + o( t ^\alpha ) ,\quad t \to 0, \quad t>0$;

\item
 $\rho(\cdot) $ is differentiable except at the origin and 
\[
 \rho'(t) =- C  \alpha  t ^{\alpha-1} + o( t^{\alpha-1}) ,\quad t \to 0, \quad t>0;
\]

 \item 
$ \rho(t) \to 0, \quad t \to + \infty.$
\end{itemize}
\end{definition} 
An example is given by  $\rho(t) = \exp\{ -t^\alpha \}$.

\medskip

Note that we can replace $X_1(\cdot)$ in the example above by any $\alpha$-process and $X_2(\cdot)$ by any differentiable process that  
satisfies condition (G) and s.t. 
\[
  \int ^{+\infty} |r'_1(t) r'_2(t)| dt 
\]
converges.
  
\subsection{Correlated processes}
Let $X_2(\cdot)$ be a process that satisfies the Geman condition (G). 
This implies that  it is differentiable in quadratic mean. 
Let $Z(t)$ be an independent stationary Gaussian process. 
We set 
\[
 X_1(t)  =  \rho_1  X'_2(t) + \rho_2  Z(t) , \quad 
 \rho_1^2  + \rho_2^2  =1.
\] 
This model is a little more restrictive than model (\ref{mgc}). 
Indeed, in (\ref{mgc}), not the whole process $Z(\cdot)$ but only its point values are independent. 
Let $r_Z(\cdot)$ be the correlation function of $Z(\cdot)$. Then,
\begin{align*}
 r_1(t) &=  - \rho_1^2 r''_2(t)  +\rho_2^2 r_Z(t), \\
 r_{12}(t) &=  \rho_2 r'_2(t).
\end{align*}
To avoid particular situations, we assume that   $\rho_i \neq 0$,
$i=1,2$.
Then we see that 
conditions for the CLT are
\[
 r_2, r'_2, r''_2, r_Z\in L^2.
\]

\subsection{Two $\alpha$-processes}
In this section  we consider two independent processes. 
The first one, $X_1(\cdot)$, is an $\alpha_1$ process (in the sense of Definition \ref{alpha}). The second one, $X_2(\cdot)$, is an $\alpha_2$ process. 
We assume that
\[
\alpha_1+\alpha_2 >2.
\]

Our goal is to prove that the number  of winding turns of $\rX(\cdot)$ 
has a finite second moment. 
Note that none of the two coordinates is 
differentiable.

Let $\psi(\cdot)$  be a compactly supported smooth enough function; 
let $\psi_\epsilon(t)=\frac 1 \epsilon \psi(\frac t \epsilon) $  and let 
$X_{2,\epsilon}(\cdot)$ be the regularization of $X_{2}(\cdot)$ 
by pathwise convolution with $\psi_\epsilon(\cdot)$. 
We denote by $N_{W,\epsilon}([0,T])$ 
the number of winding turns of $\big( X_1(\cdot),X_{2,\epsilon}(\cdot) \big)$. 
We only sketch the proof.

\medskip
 
The number of turns $N_W ([0,T)) $ is well defined and a.s. finite.
By homotopy arguments, 
\[
 N_{W}([0,T])\leq \liminf_{ \epsilon \to 0} N_{W,\epsilon}([0,T]).
\]
So we can apply Fatou's lemma to obtain 
\begin{multline*}
\E\big( N^2_{W}([0,T])\big) 
\leq
 \liminf_{\epsilon \to 0} 
 \frac {1}{\pi} \int_0^T \frac{T-t}{\sqrt{1-r_2^2(t)}} \E_c \big(X'_{2,\epsilon}(0)X'_{2,\epsilon}(t)\big) \P\{X_1(0) >0; X_1(t) >0\}  dt \\ =: T \cdot 
V_{T,\epsilon}.
\end{multline*}

At this stage, we perform  the integration by parts of Section \ref{independent} to obtain, with the obvious notation,
\begin{equation*}
T\cdot V_{T,\epsilon} = \frac {T}{\pi}  W_{T,\epsilon}  +  \frac{1}{2\pi}w_{T,\epsilon}  = 
\frac {T}{\pi}  W_{T,\epsilon} +  \frac{1}{2\pi} \bigg( \int_0^T W_{t,\epsilon} dt -W_{T,\epsilon} \bigg).
\end{equation*}
with 
\begin{equation*} 
W_{T,\epsilon}= \Bigg[\frac{\pi}{2}  - \frac{ r'_{2,\epsilon}(T)} {\sqrt{1-r^2_{2,\epsilon}(T)}}  \arccos \sqrt{\frac {1-r_1(T)}{2}} \Bigg] 
+  \int_0^T \frac{r'_{2,\epsilon}(t)}{\sqrt{1-r^2_{2,\epsilon}(t)}}  \frac{r'_1(t)}{\sqrt{1-r^2_1(t)}}dt.
\end{equation*}
Now, it is easy  to check the convergence as $\epsilon \to 0$. Eventually, we get  that 
\begin{itemize}

\item $\E\big( N^2_{W}([0,T])\big) $ is finite
\item   
\[
 \limsup_{T\to +\infty}   \frac 1 T \E\big( N^2_{W}([0,T])\big) \leq 
 \int_0^{\infty} \frac{ r'_1(t)}{ \sqrt{1-r_1^2(t)}}\frac{ r'_2(t)}{ \sqrt{1-r_2^2(t)}} dt.
\]
A direct calculation shows that the integral converges as long as 
\[
  \int ^{+\infty} r'_1(t) r'_2(t) dt \mbox{ converges.}
\]
\end{itemize}

\subsection{Non-exactly stationary processes.}
In this last example, mainly inspired by Section 4.2 of \cite{le doussal}, we consider an extension of our ideas to a class of non-stationary Gaussian processes. Assume that $\mathbf Y(t)=(Y_1(t),Y_2(t))$ is a Gaussian planar process with i.i.d. coordinates. In addition, assume that $r_{Y_1}(t,s)=f(\frac st)$ for $s\le t$ and $f$ a real function. In certain physics models, $f(x)=e^{-|\log x|^\alpha}$ is chosen, where $\alpha>1$. Define  $X_1(t)=Y_1(e^{t})$. Thus $X_1$ is a stationary and centred Gaussian process with covariance function $e^{-|t|^\alpha}$.  Put $\mathbf X(t)=(X_1(t),X_2(t))$ where $X_2$ is an independent copy of $X_1$. We only consider the case $\alpha=2$ because in this case the two coordinates are differentiable. 
Now, using a change of scale, we have the equality in law
\[
N^{\mathbf Y}_W([0,T])=N^{\mathbf X}_W([0,\log T]).
\]
Then, our results imply that
\[
\lim_{T\to\infty}\frac{\Var(N_W^{\mathbf Y}([0,T]))}{\log T} =\lim_{T\to\infty}\frac{\Var(N^{\mathbf X}_W([0,\ln T]))}{\log T} = 
\frac 1\pi \Big( \frac{\pi}{2} +   \int_0^{\infty} 
\frac{ u^2 e^{-u^2}}{1-  e^{-u^2}} du  \Big).
\]
 
 A CLT can also be obtained with  the above expression as the limit variance.
 
This can be expressed as in Section 4.2 of   \cite{le doussal} as
\[
\Var(N_W^{\mathbf Y}([0,\frac st])\approx \frac 1\pi \Big( \frac{\pi}{2} +   \int_0^{\infty} 
\frac{ u^2 e^{-u^2}}{1-  e^{-u^2}} du \Big) 
\log\Big(\frac st\Big),\ s\le t \mbox{ and }s\to\infty.
\]  
It is possible to consider also the cases $1<\alpha<2$, as in example 7.3, but the non-differentiability of the coordinates makes the procedure more involved.

\section{Auxiliary computations} \label{s:ac}

\subsection{Proof of Proposition  \ref{positive}.}

\medskip

We use the Hermite expansion \eqref{e:he}.  
Since the r.v.'s $I_q(T)$ are orthogonal for different values of $q$, 
\[
 \Var(N^*_W([0,T]))=\sum^\infty_{q=1}\Var(I_q(T)) 
 \geq \Var(I_2(T))+\Var(I_4(T)).
\]
We consider now $I_2$. 
From \eqref{e:dirac} and \eqref{e:d}, 
we know that the only non-vanishing coefficient in $I_2(T)$ is $a_0d_{1,1}=(2\pi)^{-1}$. 
Hence, 
\[
 I_2(T)=\frac{1}{2\pi\sqrt{T}}\int^T_0 X'_2(t)X_1(t)dt.
\]
Thus,
\begin{align*}
 V_\infty 
 &\geq
 \lim_{T\to\infty}\Var(I_2(T)) \\
 &= \lim_{T\to\infty} \frac{1}{4\pi^2 T}\int^T_0 2(T-t)
 \Big[ -r''_{2}(t) r_{1}(t) + r'_{12}(-t) r'_{12}(t) \Big] dt\\
 &= \frac{1}{4\pi^2}\int^\infty_{-\infty}
 \Big[ -r''_{2}(t) r_{1}(t) + r'_{12}(-t) r'_{12}(t) \Big] dt\\
 &= \frac{1}{4\pi^2}\int^\infty_{-\infty}
 \Big[ r'_{2}(t) r'_{1}(t) + r'_{12}(-t) r'_{12}(t) \Big] dt.
\end{align*}
Now, we use the Plancherel equality: 
\[
 \int^\infty_{-\infty}
 \Big[ r'_{2}(t) r'_{1}(t) + r'_{12}(-t) r'_{12}(t) \Big] dt
 = \frac{1}{4\pi^2}\int^\infty_{-\infty} \lambda^2
 \Big[ f_1(\lambda)f_2(\lambda) + |f_{12}(\lambda)|^2 \Big] d\lambda. 
\]
Bochner's matricial theorem \cite{cramer} implies that $|f_{12}|^2\leq f_1f_2$. 
Hence, $V_\infty>0$ as long as $f_1(\lambda)f_2(\lambda)>0$ with  
positive Lebesgue measure. 

Otherwise, 
note that if $f_1(\lambda)f_2(\lambda)=0$ a.e., then $f_{12}(\lambda)=0$ 
a.e., and thus $r_{12}(t)=0$ for every $t\in\R$.  
We consider $I_4$. 
Equations \eqref{e:he}, \eqref{e:dirac} and \eqref{e:d}, 
together with some routine computations, show that $I_4(T)$ 
is asymptoticaly equivalent to
\[
 \frac{1}{12\pi\sqrt{T}}\int^T_{0}\Big[ H_3(X'_2(t))X'_1(t) -X'_2(t)H_3(X_1(t))\Big]\ dt.
\]
Hence, since $r_{12}(t)=0$ for every $t$,
\begin{align*}
 \lim_{T\to\infty}\Var(I_4(T))
 &=\frac{1}{4\pi}\int^\infty_{-\infty}
 [r_1^3(t)(-r''_2(t))+r_2^3(t)(-r''_1(t))]\ dt\\
 &=\frac1{4\pi}\int^\infty_{-\infty}[f_1^{(*3)}*(\lambda^2f_2)+f_2^{(*3)}*(\lambda^2f_1)]\ d\lambda>0. 
\end{align*}
where we used the usual properties of the Fourier transform and Parseval's identity in the last equality. 
This finishes the proof.
\medskip

\subsection{ Proof of Lemma \ref{lemme de la mort}}

In the first place, we claim that
\[
\E[Z_1Z_2H_{k_3}(Z_3)H_{k_4}(Z_4)]
=q!\;\rho_{12}\rho^q_{34}+ qq!\;\rho_{13}\rho_{24}\rho^{q-1}_{34}+ qq!\;\rho_{14}\rho_{23}\rho^{q-1}_{34},
\]
for $k_3+k_4=2q$, with the convention that $\rho^{-1}_{34}=0$.

We use the Diagram formula, 
for definitions and a proof see \cite[Lem. 3.2]{taqqu}. 
The graphs 
have one vertex associated with $Z_1$, another vertex associated with $Z_2$, $k_3$ vertices associated with $Z_3$, $k_4$ vertices associated with $Z_4$, 
and they have $\frac12(2+k_3+k_4)=q+1$ edges joining the vertices associated to different r.v.'s. 
For computing the expectation we (only) need to consider the following graphs. 
(For ease of notation we write $1,2,3,4$ to represent any of the vertices associated respectively with $Z_1,Z_2,Z_3$ and $Z_4$).
\begin{itemize}
\item The first one consists in joining the vertex $1\to2$ and the vertex $3\to4$. The computation gives 
$\rho_{12}\rho^q_{34}$ but there 
are $q!$ ways to join $3\to4$. 
Thus this graph gives as contribution $\rho_{12}\rho^q_{34}q!$.

\item The second possible type of graph consists of one line $1\to3$, 
another line $2\to4$, and the remaining lines $3\to4$. 
Thus, the contribution of each array of lines is $\rho_{13}\rho_{24}\rho^{q-1}_{34}$ 
and there are $q^2(q-1)!$ of these configurations. 
Hence, the contribution in this case is   
$\rho_{13}\rho_{24}\rho^{q-1}_{34}q^2(q-1)!=\rho_{13}\rho_{24}\rho^{q-1}_{34}qq!$. 
The same can be done for the third 
graph, given $\rho_{14}\rho_{23}\rho_{34}^{q-1} qq!$.
\end{itemize}
Summing up these contributions and taking into account 
that there are no other suitable diagrams, 
the claim follows.

\medskip

Now, 
if $G\in L^2(\R,\phi(dx))$, we can expand it in terms of Hermite polynomials 
as $G=\sum_{k=0}^\infty\hat g_kH_k(x)$. 
We get
\begin{equation*}
\E[Z_1Z_2G(Z_3)G(Z_4)]=
\rho_{12}\sum_{q=0}^\infty \hat{g}^2_q\rho^q_{34}q!+\rho_{13}\rho_{24}\sum_{q=1}^\infty 
\hat{g}^2_q\rho^{q-1}_{34}qq!+\rho_{14}\rho_{23}\sum_{q=1}^\infty \hat{g}^2_q\rho^{q-1}_{34}qq!.
\end{equation*}

In our case, $G=\ind_{[0,\infty)}$ and 
from Slud \cite{Slud} we know that
\begin{equation*}
\hat g_0=\frac12;\quad 
\hat g_{2k_2+1}=\frac1{\sqrt{2\pi}}\frac{H_{2k_2}(0)}{(2k_2+1)!}
=\frac{1}{\sqrt{2\pi}}\frac{(-1)^{k_2}}{2^{k_2}k_2!(2k_2+1)}. 
\end{equation*}
Thus,
\begin{multline*}
\E[Z_1Z_2\mathbf 1_{[0,\infty)}(Z_3)\mathbf 1_{[0,\infty)}(Z_4)] \\
=\frac{\rho_{12}}{4}
+\frac{\rho_{12}}{2\pi}\sum_{j=0}^\infty\frac{(2j)!}{2^{2j}(j!)^2(2j+1)}\rho^{2j+1}_{34}
+\frac{\rho_{13}\rho_{24}+\rho_{14}\rho_{23}}{2\pi}\sum_{j=0}^\infty
\frac{(2j)!}{2^{2j}(j!)^2}\rho^{2j}_{34} \\
=\frac{\rho_{12}}{4} 
+\frac{\rho_{12}}{2\pi}  \arcsin (\rho_{34})  
+\frac{\rho_{13}\rho_{24}+\rho_{14}\rho_{23}}{2\pi} \frac 1 {\sqrt{1-\rho_{34}^2}}.
\end{multline*} 
This completes the proof of the lemma.

\section*{Acknowledgement}
We thank the reviewers for their time, their remarks and their thoughtful comments which helped to improve our paper.  
Part of this work was been done while the first author was visiting IMERL and the Centro de Matem\'atica at Montevideo.
This work has received funding from  ANR project GRAPHICS (ANR-17-CE11-0023). 
F. Dalmao and J.R. Le\'on were partially supported by Agencia Nacional de Investigaci\'on
e Innovaci\'on (ANII), Uruguay. 
F. Dalmao acknowledges CSIC's group 409.

\end{document}